\documentclass[11pt]{article}

\usepackage{amsmath,amsthm,amssymb,mathrsfs}
\usepackage{graphicx,mdframed}

\font\tencmmib=cmmib10 \skewchar\tencmmib '60
\newfam\cmmibfam
\textfont\cmmibfam=\tencmmib

\def\lessim{\ \lower4pt\hbox{$
		\buildrel{\displaystyle <}\over\sim$}\ }
\def\gessim{\ \lower4pt\hbox{$\buildrel{\displaystyle >}
		\over\sim$}\ }

\def\e{\mathbb{E}}

\newcommand{\la}{\langle}
\newcommand{\ra}{\rangle}

\newcommand{\p}{\mathbb{P}}

\newcommand{\vecx}{\mathbf{x}}
\newcommand{\veco}{\mathbf{0}}

\newtheorem{lemma}{\bf Lemma}

\newtheorem{theorem}{\bf Theorem}

\newtheorem{remark}{\bf Remark}
\newtheorem{example}{\bf Example}

\newmdtheoremenv{theo}{Theorem}

\newenvironment{Proof of lemma}{\noindent{\bf Proof of Lemma}}{\hfill$\Box$\newline}
\newenvironment{Proof of theorem}{\noindent{\it Proof of Theorem}}{\hfill\scriptsize{$\Box$}\newline}
\newenvironment{Proof of theorems}{\noindent{\bf Proof of Theorems}}{\hfill$\Box$\newline}
\newenvironment{Proof of proposition}{\noindent{\bf Proof of Proposition}}{\hfill$\Box$\newline}
\newenvironment{Proof of propositions}{\noindent{\bf Proof of Propositions}}{\hfill$\Box$\newline}
\newenvironment{Proof of exercise}{\noindent{\it Proof of Exercise:}}{\hfill$\Box$}
\begin{document}
	
		\title{Universality of Superconcentration \\ in the Sherrington-Kirkpatrick Model}

	%\author{Wei-Kuo Chen, Wai-Kit Lam}\thanks{School of Mathematics, University of Minnesota. Email: wkchen@umn.edu. W.-K. Chen's research is partially supported by NSF grant DMS-17-52184.}
	\author{Wei-Kuo Chen\thanks{University of Minnesota. Email: wkchen@umn.edu.  Partly supported by NSF grant DMS-17-52184} \and  Wai-Kit Lam \thanks{National Taiwan University. Email: waikitlam@ntu.edu.tw. Partly supported by NSTC Grant 110-2115-M002-012-MY3}}
	
	\maketitle
	
	\begin{abstract}
		%The phenomenon of superconcentration in the Sherrington-Kirkpatrick (SK) mean field spin glass model is concerned with whether one can impro
		 %It is a well-known consequence of the Gaussian-Poincar\'e inequality that the free energy with $N$ spins is of order $O(N)$. The phenomenon of %superconcentration is concerned about whether one can improve this bound.
	We study the universality of superconcentration for the free energy in the Sherrington-Kirkpatrick (SK)  model. In \cite{chatterjee_disorder_chaos}, Chatterjee showed that when the system consists of $N$ spins and Gaussian disorders, the variance of this quantity is superconcentrated by establishing an upper bound of order $N/\log{N}$, in contrast to the $O(N)$ bound obtained from the Gaussian-Poincar\'e inequality. In this paper, we show that superconcentration indeed holds for any choice of centered disorders with finite third moment, where the upper bound is expressed in terms of  an auxiliary nondecreasing function $f$ that arises in the representation of the disorder as $f(g)$ for $g$ standard normal. Under an additional regularity assumption on $f$, we further show that the variance is of order at most $N/\log{N}$.
	\end{abstract}
	
	\section{Introduction and main results}

	%Loosely speaking, a model is superconcentrated if the classical concentration techniques do not give an optimal bound on the order of the fluctuations. 
The Sherrington-Kirkpatrick (SK) model is an important mean-field spin glass that was introduced to explain some unusual magnetic behaviors of certain alloys. For a given disorder (random variable) $h$ with finite second moment, a given (inverse) temperature $\beta>0,$ and any integer $N\geq 1,$ its Hamiltonian is defined as
\begin{align*}
	-H_N(\sigma)=\frac{\beta}{\sqrt{N}}\sum_{i,j=1}^Nh_{ij}\sigma_i\sigma_j,\,\,\sigma\in\{-1,1\}^N,
\end{align*}
where $(h_{ij})_{1\leq i,j\leq N}$ are i.i.d.\ copies of $h.$ One of the main objectives in the study of the SK model is to understand the limit of the free energy,
	\begin{align*}
		F_N(\beta)=\log \sum_{\sigma\in \{-1,1\}^N}e^{-H_N(\sigma)},
	\end{align*}
which has attracted a lot of attention in physics as well as in mathematics communities, see, for instance, \cite{MPV, Panchenko, Talagrand1, Talagrand2, Talagrand3}.

This paper is concerned about the order of fluctuation for the free energy. When $h$ is standard Gaussian, it can be checked directly from the Gaussian-Poincar\'e inequality that $\mbox{Var}(F_N(\beta))=O(N)$. It is natural to ask whether one can improve this bound as $\mbox{Var}(F_N(\beta))=o(N),$ a phenomenon called superconcentration, introduced in the pioneering work of Chatterjee \cite{chatterjee_disorder_chaos,superconcentration}. In light of this notion, superconcentration was established with the bound that
for any $\beta>0,$ there exists  a constant $C=C(\beta)>0$ such that 
\begin{align}\label{add:eq-4}
	\mathrm{Var}(F_N(\beta)) \leq \frac{CN}{\log {N}},\,\,\forall N\geq 2.
\end{align} 
When $\beta\leq 1/\sqrt{2},$ much sharper bounds were also obtained in the literature.
In the case of $\beta<1/\sqrt{2},$ \cite{Aizenman} showed that $F_N(\beta)$ satisfies a central limiting theorem, and their result implies that $\mbox{Var}(F_N(\beta))=\Theta(1)$\footnote{For two nonnegative sequences $(a_N)_{N\geq 1}$ and $(b_N)_{N\geq 1},$ denote by $a_N=\Theta(b_N)$ if there exist constants $c,C>0$ such that $cb_N\leq a_N\leq Cb_N$ for all $N\geq 1.$}. At $\beta=1/\sqrt{2},$ it was predicted by \cite{Aspelmeier,PR} that a sharp phase transition should occur, namely, $\mbox{Var}(F_N(\beta))=\Theta(\log N)$. Along this conjecture, a partial result $\mbox{Var}(F_N(\beta))=O((\log N)^2)$ was known by the authors, see \cite{CL}. Interestingly, if one now considers the SK model in the presence of an external field, i.e., replacing $-H_N(\sigma)$ with $-H_N(\sigma)+r\sum_{i=1}^N\sigma_i$ for some $r>0$ in the free energy, then it was known in \cite{CDP} that the corresponding free energy obeys a central limit theorem and $\mbox{Var}(F_N(\beta))=\Theta(N)$, agreeing with the rate  obtained from the Gaussian-Poincar\'e inequality instead of exhibiting superconcentration.

While the aforementioned results addressed superconcentration assuming that the disorder is Gaussian,  we aim to investigate this phenomenon for more general choice of disorders. Note that for an arbitrary $h$ with finite second moment, the Efron-Stein inequality readily implies $\mbox{Var}(F_N(\beta))=O(N).$ In contrast to this bound, we say that the free energy is superconcentrated if $\mbox{Var}(F_N(\beta))=o(N).$ 

To state our main results, we assume that $\e h=0$ and $\e h^2=1$.  We express $h=f(g)$ for some nondecreasing function $f$ and a standard Gaussian random variable $g$. Let $g^1,g^2$ be independent copies of $g.$ For $0\leq t\leq 1,$ define
\[
	w(t)=\e f(g^1_t)f(g^2_t),
\]
where  \begin{align}
	\begin{split}
		\label{add:eq--1}
	g^1_t&=\sqrt{t}g+\sqrt{1-t}g^1,\\
	g^2_t&=\sqrt{t}g+\sqrt{1-t}g^2.
		\end{split}
\end{align} 
Note that $w(t)\to 1$ as $t\uparrow 1$ if we further assume $\e |h|^3 < \infty$. Indeed, since $f$ is nondecreasing, $f(g_t^1)f(g_t^2)$ converges to $f(g)^2$ almost surely as $t\uparrow 1$ and for any $M>0,$ denoting $E_{M,t}=\{|g_t^1|,|g_t^2|\leq M\}$, $|f(g_t^1)f(g_t^2)|\mathbf{1}_{E_{M,t}}$ is  uniformly bounded for all $t\in [0,1]$. Consequently, from the bounded convergence theorem,
\begin{align*}
	\lim_{t\uparrow 1}\e\bigr[\bigl|f(g_t^1)f(g_t^2)-f(g)^2\bigr|;E_{M,t}\bigr]=0,\,\,\forall M>0.
\end{align*}
On the other hand, from the H\"older inequality and using the union bound,
\begin{align*}
	\sup_{t\in[0,1]}\e\bigr[\bigl|f(g_t^1)f(g_t^2)\bigr|;E_{t,M}^c\bigr]&\leq \bigl(\e |f(g)|^3\bigr)^{2/3}\sup_{t\in [0,1]}\p(E_{M,t}^c)^{1/3}\\
	&\leq \bigl(\e |h|^3\bigr)^{2/3}\bigl(2\p(|g|>M)\bigr)^{1/3}\to 0\,\,\mbox{as $M\to\infty.$}
\end{align*}
These together yield the desired limit. With this, our first main result shows that superconcentration for the free energy holds for any $h$ with a finite third moment, where the upper bound for the variance is related to the rate of convergence of $w(t)$ at $1.$
\begin{theorem}\label{thm1}
There exist positive constants $c,K>0$ depending only on $\beta$ such that whenever $h$ satisfies  $\e h=0$, $\e h^2=1,$ and $\e |h|^3<\infty$, we have
	\begin{align*}
			\mathrm{Var}(F_N(\beta))\leq K\bigl(\e |h|^3+1\bigr)N\Bigl(1-w\bigl((\log N)^{-c/\log N}\bigr)+\frac{1}{\log N}\Bigr),\,\,\forall N\geq 2.
	\end{align*}
\end{theorem}

In the next main result, we let $f$ be an arbitrary absolutely continuous function and take $h=f(g)$. Note that now $f$ is not necessarily nondecreasing. If we assumption that $\e h=0$, $\e h^2=1$, and $\e|f'(g)|^3<\infty$, then we can obtain superconcentration for the free energy in the same rate as \eqref{add:eq-4}.

\begin{theorem}
	\label{thm2} There exists a constant $K>0$ depending only on $\beta$ such that whenever $h$ satisfies that $\e h=0,$ $\e h^2=1,$  and $f$ is absolutely continuous with $\e |f'(g)|^3<\infty$, we have
	\begin{align}\label{thm2:eq1}
		\mathrm{Var}(F_N(\beta))\leq K\bigl(1+\bigl(\e|f'(g)|^3\bigr)^2\bigr)\frac{N}{\log N},\,\,\forall N\geq 2.
	\end{align}
\end{theorem}

A few remarks are in position.

\begin{remark}
	\rm We do not expect to obtain the bound $N/\log N$  in Theorem~\ref{thm2} directly from Theorem~\ref{thm1}. Nevertheless, if in Theorem \ref{thm1} we assume additionally that $f$ is differentiable and $\e|f'(g)|^2<\infty$, then one can bound $1-w(t)\leq (1-t)\e |f'(g)|^2$, which follows from the mean value theorem and the fact that $w'$ is nondecreasing, see \eqref{add:eq-3} below. As a result, 
	\begin{align*}
		1-w\bigl((\log N)^{-c/\log N}\bigr)\leq \e|f'(g)|^2(1-(\log N)^{-c/\log N})\leq \e |f'(g)|^2 \frac{c\log\log{N}}{\log{N}}
	\end{align*} 
and this implies that $\mbox{Var}(F_N(\beta))=O(N\log\log N/\log N).$
\end{remark}

\begin{remark}
	\rm Under the assumption that the first four moments of $h$ agree with those of $g$ and $h$ has a finite fifth moment, one can manage to match the first and second moments of the free energies associated to $h$ and $g$ asymptotically by using the approximate Gaussian integration by parts, which will lead to $\mbox{Var}(F_N)=O(N/\log N)$ as \eqref{add:eq-4} and \eqref{thm2:eq1}, see, for example,  \cite{CNV}. Our main results address superconcentration by reducing the moment assumption as much as possible.
\end{remark}

\begin{remark}
	\rm In a more general framework, one can consider the mixed even $p$-spin model, whose Hamiltonian is defined as $$
	-H_N(\sigma)=\sum_{p=1}^\infty \frac{\beta_p}{N^{(p-1)/2}}\sum_{1\leq i_1,\ldots, i_{2p}\leq N} h_{i_1,\ldots,i_{2p}}\sigma_{i_1}\cdots\sigma_{i_{2p}},
	$$ 
	where $h_{i_1,\ldots,i_{2p}}$ for all $1\leq i_1,\ldots,i_{2p}\leq N$ and all $p\geq 1$ are i.i.d.\ copies of $h$ and $(\beta_p)_{p\geq 1}$ is a real sequence with $\sum_{p=1}^\infty 2^p\beta_p^2<\infty.$ In \cite{chatterjee_disorder_chaos},  Chatterjee showed that the corresponding free energy is again superconcentrated as long as $h$ is standard normal. We point out that the same results in Theorems \ref{thm1} and \ref{thm2} also hold in this setting by the same argument in this paper. 
\end{remark}

%\begin{remark}
%	\rm It is natural to ask whether $F_N(\beta)$ remains superconcentrated if we only assume that $\e h = 0$ and $\e h^2 = 1$. It is not clear to us whether %there is a limitation in our proof (so that it only works under a finite third moment assumption), or there is actually no superconcentration when $h$ has only %two moments. We leave this as an open problem.
%\end{remark}

We now discuss three applications of Theorems~\ref{thm1} and \ref{thm2}.

\begin{example}
	\rm If $f$ is either a polynomial or a Lipschitz function so that $\e h=0$ and $\e h^2=1,$ then Theorem \ref{thm2} holds. 
\end{example}
\begin{example}[Uniform distribution]\rm
Let $h$ be a uniform random variable on the interval $[-\sqrt{3},\sqrt{3}]$. In this case, one can write $h=f(g) = \sqrt{3}(2\Phi(g) - 1)$, where $\Phi(x):=(2\pi)^{-1/2}\int_{-\infty}^xe^{-a^2/2}~da$ for $x\in \mathbb{R}$ is the CDF of $g$. It can be checked that the assumptions in Theorem \ref{thm2} are satisfied and thus, we have $\mathrm{Var}(F_N(\beta)) =O(N/\log {N}).$
\end{example}

\begin{example}[Two-point distribution]\rm
Let $a<0<b$ and $p \in (0,1)$ satisfy $ap + b(1-p) = 0$ and $a^2p+b^2(1-p)=1$. Suppose that $\p(h = a) = p = 1 - \p(h=b)$. Note that $\e h = 0$, $\e h^2 = 1$ and $\e |h|^3 < \infty$.  We claim that
\[
\mathrm{Var}(F_N(\beta))=O\Bigl(N\sqrt{\frac{\log\log{N}}{\log{N}}}\Bigr).
\] 
To show this, we use Theorem \ref{thm1} by expressing $h=f(g)$ for
\[
f(g) = 
\begin{cases}
	a & \text{if $g\leq \Phi^{-1}(p)$,}\\
	b & \text{if $g > \Phi^{-1}(p)$},
\end{cases}
\]
where $\Phi$ is the CDF of $g.$ 
Note that $f$ is nondecreasing. Denote $\gamma=\Phi^{-1}(p).$ A direct computation shows that
\begin{align*}
	w(t) = 2ab\p(g^1_t \leq \gamma, g^2_t > \gamma) + a^2\p(g^1_t\leq \gamma, g^2_t\leq \gamma) + b^2\p(g^1_t > \gamma, g^2_t > \gamma).
\end{align*}
To compute these probabilities, we write
\begin{align*}
	\p(g^1_t \leq \gamma, g^2_t > \gamma) &= \p\bigl(\sqrt{t}g + \sqrt{1-t}g^1\leq \gamma, \sqrt{t}g + \sqrt{1-t}g^2 > \gamma\bigr)\\
	&=\iint_{\{x < y\}} \p\bigl(\sqrt{1-t}x \leq \gamma - \sqrt{t}g < \sqrt{1-t}y\bigr)\frac{e^{-(x^2+y^2)/2}}{2\pi}~dx~dy\\
	&=\iint_{\{x<y\}} \int_{\sqrt{1-t}x-\gamma}^{\sqrt{1-t}y-\gamma} \frac{e^{-z^2/(2t)}}{\sqrt{2\pi t}}~dz~\frac{e^{-(x^2+y^2)/2}}{2\pi}~dx~dy=:\Omega(t).
\end{align*}
On the other hand, 
\begin{align*}
	\p(g^1_t\leq \gamma, g^2_t\leq \gamma) &= p - \Omega(t),\\
\p(g^1_t > \gamma, g^2_t > \gamma) &= (1-p) - \Omega(t).
\end{align*}
From these and using $a^2p + b^2(1-p) = 1$, we arrive at
\begin{align*}
	w(t) &= 1 - (a-b)^2\Omega(t)\geq 1 -  \sqrt{1-t}(a-b)^2 \iint_{\{x<y\}} \frac{y-x}{\sqrt{2\pi t}} \cdot \frac{e^{-(x^2+y^2)/2}}{2\pi}~dx~dy.
\end{align*}
Thus, we can find a constant $d>0$ such that if $t$ is sufficiently close to $1$,
\[
w(t) \geq 1 - d\sqrt{1-t}.
\]
%Repeating the same calculations as in the Rademacher case, 
Theorem~\ref{thm1} then implies that for some constants $C, c>0$ and for $N$ large,
\[
\mathrm{Var}(F_N(\beta)) \leq CN\Bigl(\sqrt{1 - (\log{N})^{-c/\log{N}}} + \frac{1}{\log{N}}\Bigr)
\]
and our claim follows by noting that
\[
1 - (\log{N})^{-c/\log{N}} = 1 - \exp\Bigl(-\frac{c\log\log{N}}{\log{N}}\Bigr) \leq \frac{c\log\log{N}}{\log{N}}.
\]
\end{example}

\noindent {\bf Proof Sketch.} Our proof is based on Chatterjee's interpolation argument (see \cite{chatterjee_disorder_chaos,superconcentration}) in proving superconcentration for the free energy in the SK model with Gaussian disorder $h=g$. The argument starts by visualizing $F_N(\beta)$ as a function $F$ of i.i.d.\ standard Gaussian $\mathbf{g}=(g_{ij})_{1\leq i,j\leq N}$ and writing $\mbox{Var}(F_N(\beta))=\phi(1)-\phi(0)$, where for independent copies $\mathbf{g}^1$ and $\mathbf{g}^2$ of $\mathbf{g},$ $$\phi(t):=\e F(\mathbf{g}_t^1)F(\mathbf{g}_t^2),\,\,0\leq t\leq 1$$
for \begin{equation}
	\label{eq: interpolation_gaussian}
	\begin{split}
		\mathbf{g}_t^1&=\sqrt{t}\mathbf{g}+\sqrt{1-t}\mathbf{g}^1,\\
		\mathbf{g}_t^2&=\sqrt{t}\mathbf{g}+\sqrt{1-t}\mathbf{g}^2.
	\end{split} 
\end{equation}
The first key step uses Gaussian integration by parts inductively to show that $\psi(a)=\phi'(e^{-a})$ for $a\geq 0$ is a completely monotone function, and from the Bernstein theorem, this function can be represented as $\psi(a)=\int_{[0,\infty)}e^{-as}~\mu(ds)$ for some positive measure $\mu.$ Consequently, from H\"older's inequality, for all $0< s\leq t<1,$
\begin{align}\label{add:eq1}
	\phi'(t)\leq \phi'(s)^{\frac{\log t}{\log s}}\phi'(1)^{1-\frac{\log t}{\log s}}.
\end{align}
In the second step,  one relates $\phi'(t)$ to the second moment of the cross overlap associated to the interpolated spin system. In particular, by employing the so-called Latala argument (see \cite[Lemma 10.4]{superconcentration}), it can be shown that $\phi'(s)=O(1)$ as long as $s$ is small enough, where the fact that $g$ is symmetric was heavily used. On the other hand, the above inequality makes it possible to get $\phi'(t)=O(N^{1-\log t/\log s})$ whenever $t\geq s.$ With these, one readily obtains the desired bound $O(N/\log N)$ utilizing the relation $\mbox{Var}(F_N(\beta))=\int_0^1\phi'(t)~dt.$ 

In our argument, we adapt the following interpolation
$$\phi(t)=\e F(f(\mathbf{g}_t^1))F(f(\mathbf{g}_t^2)),\,\,0\leq t\leq 1.$$
Now the terms in $\phi'(t)$ involve $f'$ (see \eqref{add:eq-1} below). While \eqref{add:eq1} remains valid, the main difficulty arises in obtaining an useful bound for $\phi'(s)$ with small $s$. To this end, for technical purposes, we adapt the convexity argument in \cite{CL} by considering the coupled free energy \eqref{eq: Q_def} instead of using the Latala approach. Our control  in some sense relies on  an approximate Gaussian integration by parts argument throughout.

\medskip

{\noindent \bf Universality of Superconcentration in Other Models.}
Superconcentration does not only exhibit in mean-field spin glass models, but also in random growth models on the integer lattice such as first-passage percolation \cite{DHS}, directed polymers \cite{AZ}, frog model \cite{CN}.

For first-passage percolation, after a series of work \cite{BKS, BR, DHS}, it is shown that under a $2+\log$ moment assumption, the model exhibits superconcentration, and it does not depend on the distribution of the disorder. Similar results hold for many related models. In \cite{chatterjee_surface}, Chatterjee shows that superconcentration holds in a certain type of ``surface growth models'', which includes directed last-passage percolation and directed polymers, under the assumption that the disorder is a Lipschitz function of a Gaussian random variable. 

The approach to superconcentration for growth models relies on the idea in \cite{BKS}, which consists of two components: one is the $L^1$-$L^2$ bound by Talagrand (or its variants), and the other is the translation invariance of the model. Even though the superconcentration results look very similar in both mean-field spin glasses and random growth models (the upper bounds for the variances are also of order $N/\log{N}$), this approach does not seem to work in mean-field spin glass models in any obvious way, due to the fact that spin glass models and growth models are very different in nature.

\medskip

{\noindent \bf Acknowledgements.} Both authors thank S.~Bobkov for bringing \cite{Gaussian_poincare} to their attention.

\section{Proof of Theorem~\ref{thm1}}

\subsection{Some auxiliary lemmas}

In this subsection, we shall gather three elementary lemmas that will be used in our main controls later.
Let $\mathbf{g}=(g_1,\ldots,g_k)$, $\mathbf{g}^1=(g_1^1,\ldots,g_k^1),$ and $\mathbf{g}^2=(g_1^2,\ldots,g_k^2)$ be i.i.d.\ standard normal. For $0\leq t\leq 1,$ let $\mathbf{g}_t^1$ and $\mathbf{g}_t^2$ be defined as in \eqref{eq: interpolation_gaussian}.
The first lemma controls the derivative of the expectation associated to this interpolation. Although the proof has appeared in \cite{chatterjee_disorder_chaos}, we still provide a proof for completeness. We say that $F:\mathbb{R}^k\to \mathbb{R}$ is of moderate growth if $\lim_{\|\mathbf{x}\|\to \infty}|F(\mathbf{x})|e^{-a\|\mathbf{x}\|^2}=0$ for all $a>0.$ 

\begin{lemma}\label{lem-1}
	Assume that $F:\mathbb{R}^{k}\to \mathbb{R}$ is smooth and all of its partial derivatives are of moderate growth. Define 
	\begin{align*}
		\phi(t)=\e F(\mathbf{g}_t^1)F(\mathbf{g}_t^2),\,\,0\leq t\leq 1.
	\end{align*}
	Then for any $0\leq t\leq 1,$
	\begin{align}\label{lem-1:eq1}
		\phi'(t)&=\sum_{i=1}^k\e \partial_{x_i}F(\mathbf{g}_t^1)\partial_{x_i}F(\mathbf{g}_t^2)
	\end{align}
	and
	for any $0<s<t\leq 1,$ 
	\begin{align}\label{lem-1:eq2}
		\phi'(t)\leq \phi'(s)^{\frac{\log t}{\log s}}\phi'(1)^{1-\frac{\log t}{\log s}}.
	\end{align} 
\end{lemma}

\begin{proof}
	By symmetry,
	\begin{align*}
		\phi'(t)&=\sum_{i=1}^k\e \Bigl(\frac{g_{i}}{\sqrt{t}}-\frac{g_{i}^1}{\sqrt{1-t}}\Bigr)\e \partial_{x_i}F(\mathbf{g}_t^1)\cdot F(\mathbf{g}_t^2).
	\end{align*}
	Using Gaussian integration by part yields
	\begin{align*}
		\e \Bigl(\frac{g_{i}}{\sqrt{t}}-\frac{g_{i}^1}{\sqrt{1-t}}\Bigr)\e \partial_{x_i}F(\mathbf{g}_t^1)\cdot F(\mathbf{g}_t^2)&=\e \partial_{x_i}F(\mathbf{g}_t^1)\cdot \partial_{x_i}F(\mathbf{g}_t^2)
	\end{align*}
	and this gives \eqref{lem-1:eq1}. As for the second assertion, note that each term in $\phi'(t)$ is of the same form as that of $\phi(t)$. Hence, we can apply induction to show that
	\begin{align}
		\begin{split}\label{add:eq-3}
			\phi^{(n)}(t)&=\sum_{i_1,\ldots,i_n=1}^k\e \partial_{x_{i_1}\cdots x_{i_n}}F(\mathbf{g}_t^1)\partial_{x_{i_1}\cdots x_{i_n}}F(\mathbf{g}_t^2)\\
			&=\sum_{i_1,\ldots,i_n=1}^k\e \bigl[\e\bigl[\partial_{x_{i_1}\cdots x_{i_n}}F(\mathbf{g}_t^1)\big|\mathbf{g}\bigr]^2\bigr]\geq 0.
		\end{split}
	\end{align}
	Here, if we set $\psi(a)=\phi'(e^{-a})$ for $0\leq a<\infty,$ then $\psi$ is completely monotone, i.e., $(-1)^n\psi^{(n)}(a)\geq 0$ for all $0<a<\infty$ and $n\geq 0.$ From this and the Bernstein theorem (see for instance \cite[Section~XIII.4]{Bernstein}), one can express $\psi(a)=\int e^{-ax}~\mu(dx)$ for all $a>0$ by some finite positive measure $\mu$ on $[0,\infty).$ From the H\"older inequality, for any $0<a<b,$
	\begin{align*}
		\psi(a)&\leq \Bigl(\int e^{-bx}~\mu(dx)\Bigr)^{\frac{a}{b}}\Bigl(\int 1~\mu(dx)\Bigr)^{1-\frac{a}{b}}=\psi(b)^{\frac{a}{b}}\psi(0)^{1-\frac{a}{b}}
	\end{align*}
	and this is equivalent to \eqref{lem-1:eq2}.
\end{proof}

\begin{lemma}\label{add:lem2}Assume that $Y,X_1,X_2$ are random variables with finite second moment and $\e Y=0.$ Assume that $L:\mathbb{R}^2 \to \mathbb{R}$ is a differentiable function with uniformly bounded partial derivatives. Then
	\begin{align*}
		\e YL(X)&=\int_0^1 \e [\partial_{x_1}L(sX)X_1Y+\partial_{x_2}L(sX)X_2Y]~ds,
	\end{align*}
	where $X=(X_1,X_2).$
\end{lemma}

\begin{proof}
	Write
	$
	L(\mathbf{x})=L(\veco)+\int_0^1 \nabla L(s\vecx)\cdot \vecx~ds.
	$
	From this and $\e Y=0,$ putting $\mathbf{x} = X$, multiplying by $Y$ and taking expectation on both sides complete our proof.
\end{proof}

\begin{lemma}\label{add:lem1}
	Let $\psi$ be a differentiable function on $\mathbb{R}$ and be of moderate growth. Then $$
	\e |\psi(g)|\psi'(g)\leq \frac{1}{2}\e|g|\psi(g)^2.
	$$
	%where $g\sim N(0,1).$
	
\end{lemma}

\begin{proof}
	Let $\Gamma=\{x\in \mathbb{R}:\psi(x)=0\}.$ Since $\psi$ is continuous, $\Gamma$ is closed and we can write $\Gamma^c$ as a disjoint union of open intervals $\bigcup_{l\in I} (a_l,b_l)$, where $I\subseteq \mathbb{N}$ is some index set. Here, on each $(a_l,b_l)$, $\psi$ takes a fixed sign, and on $\{a_1,b_1,a_2,b_2,\ldots\}$, $\psi=0.$ From this, we can rewrite
	\begin{align}\label{add:eq3}
		\e |\psi(g)|\psi'(g)&=\sum_{l\in I}w_l\e \psi(g)\psi'(g)\mathbf{1}_{(a_l,b_l)}(g)
	\end{align}
	for some sequence $\{w_l\}_{l\in I}$ with $w_l=1$ or $-1.$ Now, we compute directly 
	\begin{align*}
		\e \psi(g)\psi'(g)\mathbf{1}_{(a_l,b_l)}(g)&=\frac{1}{\sqrt{2\pi}}\int_{a_l}^{b_l}e^{-x^2/2}\psi(x)\psi'(x)~dx\\
		&=\frac{1}{2\sqrt{2\pi}}\psi(x)^2e^{-x^2/2}\Big|_{a_l}^{b_l}+\frac{1}{2\sqrt{2\pi}}\int_{a_l}^{b_l}x\psi(x)^2e^{-x^2/2}~dx\\
		&=\frac{1}{2}\e g\psi(g)^2 \mathbf{1}_{(a_l,b_l)}(g).
	\end{align*}
	From \eqref{add:eq3}, the assertion follows.
\end{proof}

\subsection{Main controls}

Throughout this entire subsection, we assume that $f$ satisfies the following assumption,
\begin{align*}
	(A):\quad 
	\begin{split}
		&\mbox{$f$ is nondecreasing and smooth with $\e f(g)=0$, $\e f(g)^2=1$, and}\\
		&\mbox{$\e|f(g)|^3<\infty,$ and  its derivatives of all orders are of moderate growth.}
	\end{split}
\end{align*}
Let $h=f(g)$. Let $\mathbf{g}=(g_{ij})_{1\leq i,j\leq N}$, $\mathbf{g}^1=(g_{ij}^1)_{1\leq i,j\leq N}$, and $\mathbf{g}^2=(g_{ij}^2)_{1\leq i,j\leq N}$ be i.i.d.\ standard Gaussian. For any $0\leq t\leq 1,$ set 
\begin{align*}
\mathbf{g}_t^1&=\sqrt{t}\mathbf{g}+\sqrt{1-t}\mathbf{g}^1,\\
\mathbf{g}_t^2&=\sqrt{t}\mathbf{g}+\sqrt{1-t}\mathbf{g}^2.	
\end{align*}
Define
\begin{align}\label{add:eq-2}
	\phi(t)=\e \log  Z_t^1\log  Z_t^2,
\end{align}
where for $\ell = 1, 2$, 
$$
Z_t^\ell=\sum_{\sigma \in \{-1, 1\}^N}\exp \Bigl(\frac{\beta}{\sqrt{N}}\sum_{i,j=1}^Nf(g_{t,ij}^\ell)\sigma_i\sigma_j\Bigr),
$$
and $g_{t,ij}^\ell$ is the $(i,j)$-th entry of $\mathbf{g}_t^\ell$. From \eqref{lem-1:eq1},	
\begin{align}\label{add:eq-1}
	\phi'(t)=\frac{\beta^2}{N}\sum_{i,j=1}^N\e f'(g_{t,ij}^1)f'(g_{t,ij}^2)\la \sigma_i\sigma_j\tau_i\tau_j\ra_t.
\end{align}
Here, $\la \cdot \ra_t$ is the Gibbs expectation associated to the following Gibbs measure
\[
\frac{1}{Z_t^1 Z_t^2}\sum_{\sigma, \tau \in \{-1, 1\}^N} \mathbf{1}_{\{(\sigma, \tau) \in \cdot\}} \exp \Bigl(\frac{\beta}{\sqrt{N}}\sum_{i,j=1}^Nf(g_{t,ij}^1)\sigma_i\sigma_j\Bigr)\exp \Bigl(\frac{\beta}{\sqrt{N}}\sum_{i,j=1}^Nf(g_{t,ij}^2)\tau_i\tau_j\Bigr).
\]
Recall that $w(t)=\e f(g_t^1)f(g_t^2).$ From Lemma \ref{lem-1}, we also have $w'(t)=\e f'(g_t^1)f'(g_t^2)$. Finally, for $\sigma, \tau \in \{-1, 1\}^N$, we define $R(\sigma, \tau) = N^{-1}\sum_{i=1}^N \sigma_i \tau_i$.

\begin{lemma}\label{lem1}
	Under the assumption $(A)$,	we have for $0< t < 1$ that
	\begin{align*}
		%\label{lem1:eq1}
		\phi'(t)&\leq \beta^2Nw'(t)\e\bigl\la R(\sigma,\tau)^2\bigr\ra_t+4\beta^3N^{1/2}R_0(t),
	\end{align*}
	where $$R_0(t):=\frac{2^{1/3}\e |h|^{3}+1}{1-t}.$$
\end{lemma}

\begin{proof}
Let $X_1=f(g_t^1)$ and $X_2=f(g_t^2).$ Denote $U=f'(g_t^1)f'(g_t^2).$
	Since $f'$ is nonnegative,
	\begin{align*}
		\e |X_1|U&=\e\Bigl[|f(g_t^1)|f'(g_t^1)f'(g_t^2)\Bigr]\\
		&=\e\Bigl[\e_{g^1}\bigl[|f(g_t^1)|f'(g_t^1)\bigr]f'(g_t^2)\Bigr]\\
		&\leq \frac{1}{2\sqrt{1-t}}\e\Bigl[\e_{g^1}\bigl[|g^1|f(g_t^1)^2\bigr]f'(g_t^2)\Bigr]\\
		&=\frac{1}{2\sqrt{1-t}}\e \bigl[f(g_t^1)^2|g^1|f'(g_t^2)\bigr]\\
		&=\frac{1}{2(1-t)}\e \bigl[f(g_t^1)^2|g^1|g^2f(g_t^2)\bigr],
	\end{align*}
	where the inequality used Lemma \ref{add:lem1} for $\e_{g^1}\bigl[|f(g_t^1)|f'(g_t^1)\bigr]$ and the last equality used Gaussian integration by parts with respect to $g^2.$
	From the H\"older inequality and independence, it follows that
	\begin{align*}
		\nonumber	\e |X_1|U	&\leq \frac{1}{2(1-t)}\bigl(\e |f(g_t^1)^2g^2|^{3/2}\bigr)^{2/3}\bigl(\e |g^1f(g_t^2)|^{3}\bigr)^{1/3}\\
		\nonumber 	&=\frac{1}{2(1-t)}(\e |h|^{3})(\e|g|^{3/2})^{2/3}(\e |g|^3)^{1/3}\\
		%\label{add:eq-5}	
		&\leq \frac{1}{2(1-t)}\e |h|^{3}\cdot (\e |g|^3)^{2/3}=\frac{2^{1/3}}{1-t}\e |h|^{3}.
	\end{align*}
	The same bound is also valid for $\e|X_2|U.$
	From these, putting $Y = U - \e U$, and using the fact that $f'\geq 0$, we have
	\begin{align*}
		\e |X_1Y|&\leq 	\e |X_1|U+\e |X_1|\cdot \e U\leq \frac{2^{1/3}}{1-t}\e |h|^{3}+w'(t),\\
		\e |X_2Y|&\leq 	\e |X_2|U+\e |X_2|\cdot \e U\leq \frac{2^{1/3}}{1-t}\e |h|^{3}+w'(t).
	\end{align*}
	Using Gaussian integration by parts and the Cauchy-Schwarz inequality,
	\begin{align}	
		\nonumber	w'(t)&=\frac{1}{1-t}\e g^1g^2f(g_t^1)f(g_t^2)\\
		\nonumber&\leq \frac{1}{1-t}\bigl(\e |g^1f(g_t^2)|^2\bigr)^{1/2}\bigl(\e |g^2f(g_t^1)|^2\bigr)^{1/2}\\
		\label{add:eq-9}	&=\frac{1}{1-t}\bigl(\e |g^1|^2\cdot \e |f(g_t^2)|^2\bigr)^{1/2}\bigl(\e|g^2|^2\cdot\e |f(g_t^1)|^2\bigr)^{1/2}=\frac{1}{1-t}.
	\end{align}
	It follows that
	\begin{align*}
		\max\bigl(\e |X_1Y|,\e |X_2Y|\bigr)\leq \frac{2^{1/3}\e |h|^{3}+1}{1-t}=:R_0(t).
	\end{align*}
	From these and Lemma \ref{add:lem2}, for $X:=(X_1,X_2)$ and twice differentiable $L$ with $\max_{\ell=1,2}\|\partial_{x_\ell}L\|_\infty\leq \gamma,$ we arrive at
	\begin{align}
		\nonumber		\e UL(X)&=\e U\cdot \e L(X)+\e YL(X)\\
		\nonumber		&=\e U\cdot \e L(X) + \int_0^1 \e\big[\bigl(\partial_{x_1}L(sX)X_1+\partial_{x_2}L(sX)X_2\bigr)Y\big]~ds\\
		\nonumber				&\leq w'(t)\e L(X)+\gamma \bigl(\e |X_1Y|+\e |X_2Y|\bigr)\\
		\label{add:eq4}		&\leq w'(t) \e L(X)+2\gamma R_0(t).
	\end{align}
%	where we used $\e U=w'(t)$.
	Now for fixed $i,j$, conditionally on $\mathbf{g}_t^1$ and $\mathbf{g}_t^2$ except $g_{t,ij}^1$ and $g_{t,ij}^2,$ we express $\la \sigma_i\sigma_j\tau_i\tau_j\ra_t$ as $L(X)$ in distribution. A direct computation gives
	\begin{align*}
		\partial_{x_1}L(x_1,x_2)&=\frac{\beta}{\sqrt{N}}\la \sigma_i^1\sigma_j^1\tau_i^1\tau_j^1(\sigma_i^1\sigma_j^1-\sigma_i^2\sigma_j^2)\ra_t,\\
		\partial_{x_2}L(x_1,x_2)&=\frac{\beta}{\sqrt{N}}\la \sigma_i^1\sigma_j^1\tau_i^1\tau_j^1(\tau_i^1\tau_j^1-\tau_i^2\tau_j^2)\ra_t,
	\end{align*}
	where $(\sigma^1, \tau^1)$ and $(\sigma^2, \tau^2)$ are i.i.d.\ samples from the Gibbs measure associated to $\la\cdot\ra_t$. From these, $		\max_{\ell=1,2}\|\partial_{x_\ell}F\|_\infty\leq 2\beta N^{-1/2}.$
	Consequently, from \eqref{add:eq4} and using conditional expectation,
	\begin{align*}
		\e f'(g_{t,ij}^1)f'(g_{t,ij}^2)\la \sigma_i\sigma_j\tau_i\tau_j\ra_t&=\e [\e_{X_1,X_2}[UL(X)]]\leq w'(t)\e \la \sigma_i\sigma_j\tau_i\tau_j\ra_t+\frac{4\beta}{N^{1/2}}R_0(t).
	\end{align*}
	Summing these up over all $i,j$ completes our proof.
\end{proof}

\begin{lemma}\label{lem-0} 
	Assume that $(A)$ holds. 
	There exists a constant $K$ depending only on $\beta$ such that whenever $t\in [0,1)$ satisfies
	\begin{align*}
		%\label{lem-0:eq1}
		4\beta^2\log\frac{1}{1-t}<1,
	\end{align*} we have
	\begin{align*}
		\e \la R(\sigma,\tau)^2\ra_t&\leq \frac{R_1(t)}{\sqrt{N}},
	\end{align*}
	where
	\begin{align*}
		R_1(t):=K\Biggl(\log  \frac{\sqrt{2}}{\sqrt{1-4\beta^2\log \frac{1}{1-t}}}+\frac{\e |h|^3+1}{1-t}+\e|h|^3\Biggr).
	\end{align*}
\end{lemma}

\begin{proof}
	For $t\in [0,1]$ and $\lambda\geq 0,$ consider
\begin{equation}
	\label{eq: Q_def}
Q(t,\lambda):=\e \log \sum_{\sigma,\tau\in \{-1,1\}^N}\exp\Bigl(\frac{\beta}{\sqrt{N}}\sum_{i,j=1}^N\bigl(f(g_{t,ij}^1)\sigma_i\sigma_j+f(g_{t,ij}^2)\tau_i\tau_j\bigr)+\lambda \beta^2 NR(\sigma,\tau)^2\Bigr).
\end{equation}
Denote by $\la \cdot\ra_{t,\lambda}$ the Gibbs average with respect to the i.i.d.\ samples $(\sigma^\ell,\tau^\ell)_{\ell\geq 1}$ from the Gibbs measure associated to this free energy $Q$. A direct differentiation and Gaussian integration by parts yield that
\begin{align*}
	\partial_tQ(t,\lambda)&=\frac{\beta}{2\sqrt{N}}\sum_{i,j=1}^N \e\Bigl\la \Bigl(\frac{g_{ij}}{\sqrt{t}}-\frac{g_{ij}^1}{\sqrt{1-t}}\Bigr)f'(g_{t,ij}^1)\sigma_i\sigma_j+\Bigl(\frac{g_{ij}}{\sqrt{t}}-\frac{g_{ij}^2}{\sqrt{1-t}}\Bigr)f'(g_{t,ij}^2)\tau_i\tau_j\Big\ra_{t,\lambda}\\
	&=\frac{\beta^2}{N}\sum_{i,j=1}^N\e f'(g_{t,ij}^1)f'(g_{t,ij}^2)\bigl\la\sigma_i^1\sigma_j^1\bigl(\tau_i^1\tau_j^1-\tau_i^2\tau_j^2\bigr)\bigr\ra_{t,\lambda}.
\end{align*}
In the same manner as the proof of Lemma \ref{lem1}, if we let $X_1=f(g_{t,ij}^1),$ $X_2=f(g_{t,ij}^2),$ and $U=f'(g_{t,ij}^1)f'(g_{t,ij}^2)$, we can express $\bigl\la\sigma_i^1\sigma_j^1\bigl(\tau_i^1\tau_j^1-\tau_i^2\tau_j^2\bigr)\bigr\ra_{t,\lambda}$ as $L(X_1,X_2)$ in distribution. In this case,
\begin{align*}
	\partial_{x_1}L(x_1,x_2)&=\frac{\beta}{\sqrt{N}}\bigl\la\sigma_i^1\sigma_j^1\bigl(\tau_i^1\tau_j^1-\tau_i^2\tau_j^2\bigr)(\sigma_i^1\sigma_j^1+\sigma_i^2\sigma_j^2-2\sigma_i^3\sigma_j^3)\bigr\ra_{t,\lambda},\\
	\partial_{x_2}L(x_1,x_2)&=\frac{\beta}{\sqrt{N}}\bigl\la\sigma_i^1\sigma_j^1\bigl(\tau_i^1\tau_j^1-\tau_i^2\tau_j^2\bigr)(\tau_i^1\tau_j^1+\tau_i^2\tau_j^2-2\tau_i^3\tau_j^3)\bigr\ra_{t,\lambda}.
\end{align*}
Therefore, from \eqref{add:eq4}, for $D(t):=16\beta^3R_0(t),$
we have
\begin{align*}
	\partial_tQ(t,\lambda)&\leq \beta^2Nw'(t)\e \bigl\la R(\sigma^1,\tau^1)^2-R(\sigma^1,\tau^2)^2\bigr\ra_{\lambda,t}+\sqrt{N}D(t).
\end{align*}
From this, whenever $0\leq w(t)\leq \lambda,$
\begin{align*}
	\partial_t\bigl(Q(t,\lambda-w(t))\bigr)&=\partial_tQ(t,\lambda-w(t))-w'(t)\partial_\lambda Q(t,\lambda-w(t))\\
	&\leq -\beta^2Nw'(t)\e\bigl\la R(\sigma^1,\tau^2)^2\bigr\ra_{\lambda,t}+\sqrt{N}D(t)\leq \sqrt{N}D(t),
\end{align*}
which implies that
\begin{align*}
	Q(t,\lambda-w(t))-Q(0,\lambda)=Q(t,\lambda-w(t))-Q(0,\lambda-w(0))&\leq \sqrt{N}D(t).
\end{align*}
Let $\lambda_0>0$ such that $2\beta^2\lambda_0<1$. For any $t\geq 0$ satisfying $2\beta^2(\lambda_0+w(t))<1$, if we plug $\lambda=\lambda_0+w(t)$ into the above inequality, then
\begin{align}\label{add:eq5}
	Q(t,\lambda_0)\leq Q(0,\lambda_0+w(t))+\sqrt{N}D(t).
\end{align} Finally, since $Q(t,\lambda)$ is convex in $\lambda$,
\begin{align}
\nonumber	\lambda_0 \beta^2N\e \la R(\sigma,\tau)^2\ra_t&\leq \lambda_0\partial_\lambda Q(t,0)\\
\nonumber	&\leq Q(t,\lambda_0)-Q(t,0)\\
\label{add:eq-13}	&\leq Q(0,\lambda_0+w(t))-Q(0,0)+\sqrt{N}D(t),
\end{align}
where the third inequality used \eqref{add:eq5} and the fact that $Q(t, 0) = Q(0, 0)$. 

In order to bound the right-hand side of the last inequality, our next step is to show that we can essentially replace $f(g_{ij}^1)$ and $f(g_{ij}^2)$ in $Q(0,\cdot)$ by i.i.d.\ standard normal random variables by using approximate Gaussian integration by parts. Let $(z_{ij}^1)_{1\leq i,j\leq N}$ and $(z_{ij}^2)_{1\leq i,j\leq N}$ be i.i.d.\ standard normal independent of $\mathbf{g},\mathbf{g}^1,\mathbf{g}^2.$ Define
\begin{align*}
	\rho(\lambda)=\e \log \sum_{\sigma,\tau\in \{-1,1\}^N}\exp\Bigl(\frac{\beta}{\sqrt{N}}\sum_{i,j=1}^N\bigl(z_{ij}^1\sigma_i\sigma_j+z_{ij}^2\tau_i\tau_j\bigr)+\lambda\beta^2 NR(\sigma,\tau)^2\Bigr).	
\end{align*}
This is essentially the same as $Q(0,\lambda)$ with the replacement of $(f(g_{ij}^1),f(g_{ij}^2))$ by $(z_{ij}^1,z_{ij}^2).$ Denote $w_{ij}^1=f(g_{ij}^1)$ and $w_{ij}^2=f(g_{ij}^2).$ For any $0\leq s\leq 1,$ set the interpolated free energy
\begin{align*}
	\rho(s,\lambda)&=\e \log \sum_{\sigma,\tau\in \{-1,1\}^N}\exp\Bigl(\frac{\beta}{\sqrt{N}}\sum_{i,j=1}^N\bigl((\sqrt{s}w_{ij}^1+\sqrt{1-s}z_{ij}^1)\sigma_i\sigma_j\\
	&\qquad\qquad\qquad\qquad\qquad\qquad\qquad+(\sqrt{s}w_{ij}^2+\sqrt{1-s}z_{ij}^2)\tau_i\tau_j\bigr)+\lambda \beta^2NR(\sigma,\tau)^2\Bigr).
\end{align*}
Similar to the Gibbs expectation $\la \cdot\ra_{t,\lambda},$ we let $\la \cdot\ra_s'$ be the Gibbs expectation with respect to the i.i.d.\ $(\sigma^\ell,\tau^\ell)_{\ell\geq 1}$ sampled from the Gibbs measure associated to the free energy $\rho(s,\lambda).$
It follows that
\begin{align*}
	\frac{\partial}{\partial s}\rho(s,\lambda)&=\frac{\beta}{2\sqrt{Ns}}\sum_{i,j=1}^N \e\bigl\la w_{ij}^1\sigma_i\sigma_j+w_{ij}^2\tau_i\tau_j\big\ra_{s}'-\frac{\beta}{2\sqrt{N(1-s)}}\sum_{i,j=1}^N \e\bigl\la z_{ij}^1\sigma_i\sigma_j+z_{ij}^2\tau_i\tau_j\big\ra_{s}'.
\end{align*}
Here, the second term can be computed by the usual Gaussian integration by parts,
\begin{align*}
	&\frac{\beta}{2\sqrt{N(1-s)}}\sum_{i,j=1}^N \e\bigl\la z_{ij}^1\sigma_i\sigma_j+z_{ij}^2\tau_i\tau_j\big\ra_{s}'=\frac{\beta^2N}{2}\e\bigl\la 2-R(\sigma^1,\sigma^2)^2-R(\tau^1,\tau^2)^2\bigr\ra_s'.
\end{align*}
As for the first term, note that $\e h=0,\e h^2=1,$ and $\e |h|^3<\infty$, we can use approximate Gaussian integration by parts (see, e.g., \cite[Lemma~2.2]{AC16}) to obtain
\begin{align*}
	&\Bigl|\e\bigl\la w_{ij}^1\sigma_i\sigma_j+w_{ij}^2\tau_i\tau_j\big\ra_{s}'-\frac{\beta\sqrt{s}}{\sqrt{N}}\e \bigl\la (2-\sigma_i^1\sigma_j^1\sigma_i^2\sigma_j^2-\tau_i^1\tau_j^1\tau_i^2\tau_j^2)\bigr\ra_s'\Bigr|\leq \frac{3}{2}\cdot \frac{8s\beta^2}{N}\cdot\e |h|^3=\frac{12 \beta^2s}{N}\e|h|^3.
\end{align*}
Summing over all $i,j$ yields that
\begin{align*}
	&\Bigl|\frac{\beta}{2\sqrt{Ns}}\sum_{i,j=1}^N \e\bigl\la w_{ij}^1\sigma_i\sigma_j+w_{ij}^2\tau_i\tau_j\big\ra_{s}'-\frac{\beta^2N}{2}\e\bigl\la 2-R(\sigma^1,\sigma^2)^2-R(\tau^1,\tau^2)^2\bigr\ra_s'\Bigr|\leq 6\beta^3\sqrt{N}\e|h|^3.
\end{align*}
Consequently, we arrive at
\begin{align*}
	|\partial_s\rho(s,\lambda)|&\leq 6\beta^3\sqrt{N}\e|h|^3,
\end{align*}
which implies that
\begin{align*}
	|Q(0,\lambda)-\rho(0, \lambda)|=|\rho(1,\lambda)-\rho(0,\lambda)|\leq 6\beta^3\sqrt{N}\e|h|^3.
\end{align*}
This together with \eqref{add:eq-13} implies that
\begin{align*}
	\lambda_0\beta^2 N\e \la R(\sigma,\tau)^2\ra_t
	&\leq Q(0,\lambda_0+w(t))-Q(0,0)+\sqrt{N}D(t)\\
	&\leq \rho(0,\lambda_0+w(t))-\rho(0,0)+\sqrt{N}\bigl(12\beta^3\e|h|^3+D(t)\bigr).
\end{align*}

In the last step, note that
\begin{align*}
	\rho(0,\lambda_0+w(t))-\rho(0,0)&=\e\log  \bigl\la \exp \beta^2(\lambda_0+w(t))NR(\sigma,\tau)^2\bigr\ra_{0}'\\
	&\leq \log  \e \bigl\la \exp \beta^2(\lambda_0+w(t))NR(\sigma,\tau)^2\bigr\ra_{0}'.
\end{align*}
Observe that due to the symmetry of $(z_{ij}^1)$ and $(z_{ij}^2)$, under the expectation $\e\la \cdot\ra_{0}'$, $NR(\sigma,\tau)$ equals $X_1+\cdots+X_N$ for $X_1,\ldots,X_N$ i.i.d.\ Rademacher$(1/2)$ random variables in distribution. Consequently,
as long as 
$$
2\beta^2\Bigl(\lambda_0+\log \frac{1}{1-t}\Bigr)<1,
$$
we have
\begin{align*}
	\e\bigl\la \exp \beta^2(\lambda_0+w(t))NR(\sigma,\tau)^2\bigr\ra_{0}'&=\e \exp\Bigl(\frac{\beta^2(\lambda_0+w(t))}{N}\Bigl(\sum_{i=1}^NX_i\Bigr)^2\Bigr)\\
	&\leq \frac{1}{\sqrt{1-2\beta^2(\lambda_0+w(t))}}\\
	&\leq \frac{1}{\sqrt{1-2\beta^2\bigl(\lambda_0+\log\frac{1}{1-t}\bigr)}},
\end{align*}
where recalling \eqref{add:eq-9}, the last inequality used  the bound
\begin{align*}
	w(t)=\int_0^tw'(s)~ds\leq \log \frac{1}{1-t}.
\end{align*}
It follows that
\begin{align}\label{add:eq-6}
	\e \la R(\sigma,\tau)^2\ra_t
	&\leq \frac{1}{N\lambda_0\beta^2}\log  \frac{1}{\sqrt{1-2\beta^2\bigl(\lambda_0+\log \frac{1}{1-t}\bigr)}}+\frac{12\beta^3\e|h|^3+D(t)}{\sqrt{N}\lambda_0\beta^2}.
\end{align}
Recalling that $$D(t)=16\beta^3R_0(t)=\frac{16\beta^3}{1-t}\bigl(2^{1/3}\e|h|^3+1\bigr)
$$ and taking $\lambda_0=1/(4\beta^2)$, whenever $t$ satisfies
$$
4\beta^2\log \frac{1}{1-t}<1,
$$
we have
\begin{align*}
	\e \la R(\sigma,\tau)^2\ra_t&\leq \frac{4}{N}\log  \frac{\sqrt{2}}{\sqrt{1-4\beta^2\log \frac{1}{1-t}}}+\frac{64\beta^3\bigl(2^{1/3}\e|h|^3+1\bigr)}{\sqrt{N}(1-t)}+\frac{48\beta^3\e|h|^3}{\sqrt{N}}.
\end{align*}
This completes our proof.
\end{proof}

\subsection{Proof of Theorem \ref{thm1}}

 {\bf\noindent Smooth Case:}  First, we show that Theorem \ref{thm1} holds under the assumption $(A)$. Recall $\phi$ from \eqref{add:eq-2}. It suffices to bound $\phi(1)-\phi(0).$
	Denote $\eta (t)=\e\la R(\sigma,\tau)^2\ra_t$. For $0<r<1,$ write
	\begin{align*}
		\phi(1)-\phi(0)&=\int_0^r \phi'(t)~dt + \int_r^1\phi'(t)~dt.
	\end{align*}
	By Lemma~\ref{lem1} and integration by parts, the first term is bounded above by
	\begin{align*}
		\int_0^r\phi'(t)~dt&\leq \beta^2N\int_0^r w'(t)\eta (t)~dt+4\beta^3N^{1/2}\int_0^r R_0(t)~dt\\
		&=\beta^2 N\Bigl(w(r)\eta(r) - w(0)\eta(0) - \int_0^r w(t)\eta '(t)~dt\Bigr)+4\beta^3N^{1/2}\int_0^r R_0(t)~dt\\
		&\leq \beta^2 N\eta (r)+4\beta^3N^{1/2}\int_0^r R_0(t)~dt,
	\end{align*}
	where we dropped $w(0)\eta (0)$ and $\int_0^r w(t)\eta'(t)~dt$ since they are both nonnegative due to \eqref{add:eq-3}.
%	Since $\phi'(t)$ is nondecreasing by \eqref{add:eq-3}, using \eqref{lem1:eq1} and \eqref{add:eq-9}, the first term is bounded by $$
%	\int_0^s\phi'(t)dt\leq \phi'(s) \leq \beta^2Nw'(s)\eta (s) + 4\beta^3N^{1/2} R_0(s) \leq \frac{\beta^2N\eta(s) + 4\beta^3 N^{1/2} R_0(s)}{1-s}.
%	$$ 
	
	On the other hand, observe that $\phi'(t)\leq \beta^2 Nw'(t).$
	It follows that
	\begin{align*}
		\int_r^1\phi'(t)~dt&\leq \beta^2 N(w(1)-w(r)).
	\end{align*}
	Combining these together yields that
	\begin{align*}
		\begin{split}%\label{add:eq-10}
		\phi(1)-\phi(0)&\leq\beta^2N\Bigl(\eta(r)+\frac{4\beta}{N^{1/2}}\int_0^r R_0(t)~dt+w(1)-w(r)\Bigr).
		\end{split}
	\end{align*}

	Next, from Lemma~\ref{lem-0}, we fix $0<s<1$ such that $$\eta (s)\leq \frac{R_1(s)}{\sqrt{N}}.$$  
	By using \eqref{lem-1:eq2}, for any $r$ satisfying $s\leq r< 1,$
	\begin{align}\label{add:eq2}
		\eta (r)&\leq \bigl(	\eta (s)\bigr)^{\frac{\log r}{\log s}}\bigl(\eta (1)\bigr)^{1-\frac{\log r}{\log s}}\leq \eta (s)^{\frac{\log r}{\log s}}\leq N^{-\frac{\log r}{2\log s}}R_1(s)^{\frac{\log r}{\log s}}\leq N^{-\frac{\log r}{2\log s}}(1+R_1(s)).
	\end{align}
	Consequently, there exists some $K'>0$ depending only on $\beta$ such that for any $r$ satisfying $s\leq r< 1,$
	\begin{align}
	\nonumber	\phi(1)-\phi(0)&\leq \beta^2N\Bigl(N^{-\frac{\log r}{2\log s}}(1+R_1(s))+4\beta N^{-\frac{1}{2}}\int_0^r R_0(t)~dt+w(1)-w(r)\Bigr)\\
	\label{eq1}	&\leq K'(\e|h|^3+1)N\Bigl(N^{-\frac{1}{2}}+N^{-\frac{\log r}{2\log s}}+N^{-\frac{1}{2}} \int_0^r \frac{dt}{1-t}+w(1)-w(r)\Bigr).
	\end{align}
To control the right-hand side,  let $N\geq 2$ and take
\begin{align*}
	r=(\log N)^{\frac{2\log s}{\log N}}.
\end{align*}
Note that $s\leq r\leq 1$ and that if $a=1-r,$ then $1-a=(\log N)^{{2\log s}/{\log N}}=(\log N)^{-{2\log (s^{-1})}/{\log N}}$. Using the bound $1-cx\leq (1-x)^c$ for all $x\in [0,1]$ and $c\geq 1$ implies that 
\begin{align*}
1-\frac{a\log N}{2\log (s^{-1})}\leq 	(1-a)^{\frac{\log N}{2\log (s^{-1})}}=\frac{1}{\log N}.
\end{align*}
It follows that
\begin{align*}%\label{add:eq-7}
	1-r&=a\geq \frac{2\log (s^{-1})}{\log N}\Bigl(1-\frac{1}{\log N}\Bigr)
\end{align*}
and thus, there exists some $C$ depending only on $s$ such that
\begin{align}\label{add:eq-7}
\int_{0}^r\frac{dt}{1-t} = \log \frac{1}{1-r}\leq C\log\log N.
\end{align}
On the other hand, from our choice of $r,$
\begin{align}\label{add:eq-8}
	N^{-\frac{\log r}{2\log s}}&=\frac{1}{\log N}.
\end{align}
Putting \eqref{add:eq-7} and \eqref{add:eq-8} back to \eqref{eq1} yields that
	\begin{align*}
		%\label{add:eq103}
		\mbox{Var}(F_N(\beta))=\phi(1)-\phi(0)&\leq K''\bigl(\e|h|^3+1\bigr)N\Bigl(w(1)-w(r)+\frac{1}{\log N}\Bigr),
	\end{align*}
where $K''$ is a constant depending only on $\beta.$ This proves Theorem~\ref{thm1} under assumption $(A)$.

\medskip

{\noindent \bf General Case:} Assume that $h$ satisfies $\e h=0,\e h^2 = 1$ and $\e|h|^3<\infty$ and $h$ can be written as $h=f(g)$ for some nondecreasing $f$, where $g$ is a standard normal random variable. For any integer $n\geq 1,$ set $f_n(x)=\max(\min(f(x),n),-n).$ Let $h_n=\bar f_n(g)$ for
$$
\bar f_n(x):=\frac{f_n(x)-\e f_n(g)}{\sqrt{\mbox{Var}(f_n(g))}}.
$$
Note that $\e h_n=0$ and $\e h_n^2=1$. Also, we have $|f_n(g)| \leq |f(g)|$ for all $n$, and hence by the dominated convergence theorem, $\e |f_n(g) - f(g)|^3 \to 0$ as $n\to\infty$. Thus, $\e |h_n-h|^3\to 0$ and $\e\bar f_n(g_t^1)\bar f_n(g_t^2) \to w(t)$ as $n\to\infty.$ From these, if we can show that $h_n$ enjoys the inequality in Theorem \ref{thm1}, then so does $h.$ To this end, for any fixed $n,$ since $\bar f_n$ is bounded and nondecreasing, we can construct a sequence of smooth and nondecreasing functions $(\bar f_{n,k})_{k\geq 1}$ of moderate growth (for instance, take $\bar f_{n,k}(x)=\e \bar f_n(x+g/\sqrt{k})$) so that $\bar f_{n,k}$ satisfies the condition $(A)$ and for $h_{n,k}:=\bar f_{n,k}(g)$, $\e|h_n-h_{n,k}|^3\to 0$ as $k\to \infty.$ Since $h_{n,k}$ satisfies the upper bound in Theorem~\ref{thm1} for any $k\geq 1$, we can pass to the limit $k\to\infty$ to obtain the same bound for $h_n,$ completing our proof.

\section{Proof of Theorem \ref{thm2}}

{\noindent \bf Smooth Case}: Assume that $f$ satisfies the extra assumption that $f$ is smooth and its derivatives of all orders are of moderate growth.
Recall $\phi(t)$ from \eqref{add:eq-2}. Note that in the proof of Lemma \ref{lem1}, we can bound $$
	\e |X_1|U\leq (\e|h|^3)^{1/3}(\e|f'(g_t^1)f'(g_t^2)|^{3/2})^{2/3} \leq (\e |h|^3)^{1/3}(\e|f'(g)|^3)^{2/3}$$ by the H\"older inequality. Thus, the statement of Lemma~\ref{lem1} is valid with the replacement of $R_0(t)$ by the constant $$C_0:=(\e|h|^3)^{1/3}(\e|f'(g)|^3)^{2/3}+\e f'(g)^2,$$ that is,
\begin{align}\label{add:eq-11}
	\phi'(t)&\leq \beta^2Nw'(t)\eta(t)+\beta^3N^{1/2}C_0,
\end{align}
where we recall that $\eta(t)=\e\la R(\sigma,\tau)^2\ra_t.$ 
From this bound, it can also be checked directly that \eqref{add:eq-6} holds with $D(t)$ being replaced by $D_0=16\beta^3C_0$. Moreover, as long as $$
2\beta^2\Bigl(\lambda_0+\log \frac{1}{1-t}\Bigr)<1,
$$
we have
	\begin{align*}
	\eta(t)
	&\leq \frac{1}{N\lambda_0\beta^2}\log  \frac{1}{\sqrt{1-2\beta^2\bigl(\lambda_0+\log \frac{1}{1-t}\bigr)}}+\frac{12\beta^3\e|h|^3+D_0}{\sqrt{N}\lambda_0\beta^2}.
\end{align*}
Letting $\lambda_0=1/(4\beta^2)$, this inequality then implies that whenever 
$$
4\beta^2\log \frac{1}{1-t}<1,
$$
we have
\begin{align}\label{add:eq-12}
\eta(t)
&\leq \frac{C_1(t)}{\sqrt{N}}
\end{align}
for
\begin{align*}
C_1(t)=4\log \frac{\sqrt{2}}{\sqrt{1-4\beta^2\log \frac{1}{1-t}}}+48\beta^3\e|h|^3+4D_0.
\end{align*}
Now, by using \eqref{add:eq-11}, for any $0<s<1,$
	\begin{align*}
	\mbox{Var}(F_N(\beta))&=\phi(1)-\phi(0)=\int_0^1\phi'(t)~dt\\
	&\leq \int_0^1(\beta^2Nw'(t)\eta(t)+\beta^3\sqrt{N}C_0)~dt\\
		&\leq \beta^2Nw'(1)\int_0^1\eta(t)~dt+\beta^3\sqrt{N}C_0\\
&\leq \beta^2N\e f'(g)^2\Bigl(\eta(s)+\int_s^1\eta(t)~dt\Bigr)+\beta^3\sqrt{N}C_0,
\end{align*}
where the last inequality used monotonicity of $\eta$.
Here, we can select and fix $s$ satisfying that $4\beta^2\log (1-s)^{-1}<1$ so that we can apply \eqref{add:eq-12} to bound $\eta(s)\leq C_1(s)N^{-1/2}$. In a similar manner as that of \eqref{add:eq2}, we can bound that for any $s\leq r\leq 1,$
\begin{align*}
	\eta(r)\leq N^{-\frac{\log r}{2\log s}}(1+C_1(s)),
\end{align*}
which implies that
\begin{align*}
	\int_s^1\eta(r)~dr\leq(1+C_1(s)) \int_s^1N^{-\frac{\log r}{2\log s}}~dr \leq 2(1+C_1(s))\frac{\log (s^{-1})}{\log N},
\end{align*}
where the second inequality used the bound that for any $x>1,$
$$\int_s^1 x^{\log r}~dr=\frac{1-s^{1+\log x}}{1+\log x}\leq \frac{1}{1+\log x}\leq \frac{1}{\log x} .$$
Putting these together, we arrive at
\begin{align}
\nonumber	\mbox{Var}(F_N(\beta))&\leq \beta^2N\e f'(g)^2\Bigl(\frac{C_1(s)}{\sqrt{N}}+2(1+C_1(s))\frac{\log (s^{-1})}{\log N}\Bigr)+\beta^3\sqrt{N}C_0\\
		\nonumber&\leq K\e f'(g)^2\bigl(1+\e|h|^3+\e|f'(g)|^2+(\e|h|^3)^{1/3}(\e|f'(g)|^3)^{2/3}\bigr)\frac{N}{\log N}
\end{align}
for some universal constant $K$ depending only on $\beta.$ Note that the following Gaussian-Poincar\'e inequality holds (see, e.g., \cite[Eq.~(2.5)]{Gaussian_poincare}),
$$
	\e |h|^3=\e |f(g)|^3 \leq C\e |f'(g)|^3,
$$
where $C$ is a universal constant independent of $f$. We can bound each $\e|h|^3$ in our main control above by $\e|f'(g)|^3$. Together with the trivial bound $\e f'(g)^2\leq 1+\e |f'(g)|^3$, we obtain the desired inequality \eqref{thm2:eq1}.

\medskip

{\noindent \bf General Case:} We continue to handle the general case in Theorem~\ref{thm2}. First of all, we argue that without loss of generality, we can assume that $f$ is uniformly bounded on $\mathbb{R}.$ Indeed, consider the absolutely continuous function $f_M=\max(-M,\min(M,f))$ for $M\geq 1.$ We see that $|f_M(x)|\leq |f(x)|$ for all $x$ and $|f_M'(x)|\leq |f'(x)|$ a.e. Since $\e|f(g)|^3$ and $\e |f'(g)|^3$ are both finite, if we define
\begin{align*}
	\bar f_M(x)=\frac{f_M(x)-\e f_M(g)}{\sqrt{\mbox{Var}(f_M(g))}},
\end{align*}
then $h_M:=\bar f_M(g)$ satisfies the assumption in Theorem \ref{thm2}. On the other hand, by the dominated convergence theorem, we also have that $\e|f_M(g)-f(g)|^3\to 0$ and $\e|f_M'(g)-f'(g)|^3\to 0$, which in turn implies that $\e|h_M-h|^3\to 0$ and $\e|\bar f_M'(g)-f'(g)|^3\to 0.$ Hence, in proving Theorem \ref{thm2}, we shall further assume that $f$ is uniformly bounded from now on.

Let $(a_n)$ and $(b_n)$ be two real sequences with $a_n<b_n$, $a_n\to-\infty$, and $b_n\to\infty.$ For each $n,$ let $f_n$ be an absolute continuous function defined as $f_n\equiv f$ on $[a_n,b_n],$ $f_n\equiv 0$ outside $[a_n-1,b_n+1]$, and linear otherwise. Since $\e|f(g)|^3$ and $\e|f'(g)|^3$ are both finite and $f$ is uniformly bounded, it can be checked that 
\begin{align}
	\begin{split}
		\label{add:eq10}
\lim_{n\to\infty}\e |f_n(g)-f(g)|^3&= 0,\\
\lim_{n\to\infty}\e|f_n'(g)-f'(g)|^3&=0.
	\end{split}
\end{align}
In addition, because $f_n$ is compactly supported, $|f_n|^3$ is integrable on $\mathbb{R}$ with respect to the Lebesgue measure. Since a.e.
\begin{align*}
	f_n'(x)=f'(x)\mathbf{1}_{[a_n,b_n]}+f(a_n)\mathbf{1}_{[a_n-1,a_n)}-f(b_n)\mathbf{1}_{(b_n,b_n+1]},
\end{align*}
we also have 
\begin{align*}
	\int_{-\infty}^\infty |f_n'(x)|^3~dx\leq |f(a_n)|^3+|f(b_n)|^3+\sqrt{2\pi}e^{(a_n^2+b_n^2)/2}\e |f'(g)|^3<\infty.
\end{align*}
With these, for any $n\geq 1,$ there exists a sequence of smooth functions $(\phi_{n,k})_{k\geq 1}$ with compact support such that $\phi_{n,k}\to f_n$ and $\phi_{n,k}'\to f_n'$ as $k\to\infty$ under the $L^3$-norm with respect to the Lebesgue measure on $\mathbb{R}$ (see, for instance, \cite[Corollary~3.23]{Sobolev}). This readily implies that 
\begin{align}
	\begin{split}
		\label{add:eq11}
		\lim_{k\to\infty}\e|\phi_{n,k}(g)-f_n(g)|^3=0,\\
		\lim_{k\to\infty}\e|\phi_{n,k}'(g)-f_n'(g)|^3= 0.
	\end{split}
\end{align}
Now, let  $$
\bar f_{n,k}(x)=\frac{\phi_{n,k}(x)-\e\phi_{n,k}(g)}{\sqrt{\mbox{Var}(\phi_{n,k}(g))}}.
$$
From \eqref{add:eq10}, \eqref{add:eq11}, $\e f(g)=0,$ and $\e f(g)^2=1,$
\begin{align*}
\lim_{n\to\infty}	\lim_{k\to\infty}\e|\bar f_{n,k}(g)-f(g)|^3&=0,\\
\lim_{n\to\infty}	\lim_{k\to\infty}\e|\bar f_{n,k}'(g)-f'(g)|^3&=0.
\end{align*}
Here, the first limit readily implies that the variance of the free energy associated to $\bar f_{n,k}(g)$ converges to that associated to $h$ in the limit $k\to\infty$ and then $n\to\infty,$
while the second limit leads to
$$\lim_{n\to\infty}\lim_{k\to\infty}\e|\bar f_{n,k}(g)|^3=\e|f'(g)|^3.$$
Since $h_{n,k}:= \bar{f}_{n,k}(g)$ satisfies all the assumptions in Theorem~\ref{thm2} and the derivatives of $\bar{f}_{n,k}$ of all orders are of moderate growth by the compact supportiveness of $\phi_{n,k}$, from the smooth case above, the inequality \eqref{thm2:eq1} holds for $h_{n,k},$ from which sending the limit in the order $k\to\infty$ and then $n\to\infty$ completes our proof.

\end{document}